\documentclass[psamsfonts,reqno,11pt]{amsart}
\usepackage{amsthm, amssymb}
\usepackage{tikz-cd}
\usepackage{multirow}
\usepackage[all]{xy}
\usepackage{color, soul}

\newtheorem{theorem}{Theorem}[section]
\newtheorem{lemma}[theorem]{Lemma}

\newtheorem{example}[theorem]{Example}
\newtheorem{proposition}[theorem]{Proposition}
\newtheorem{corollary}[theorem]{Corollary}
\newtheorem{definition}[theorem]{Definition}
\newtheorem{nono-theorem}{Theorem}
\newtheorem*{theorem*}{Theorem}
\theoremstyle{definition}

\DeclareSymbolFont{AMSb}{U}{msb}{m}{n}

\DeclareMathSymbol{\A}{\mathbin}{AMSb}{"41}
\DeclareMathSymbol{\B}{\mathbin}{AMSb}{"42}
\DeclareMathSymbol{\C}{\mathbin}{AMSb}{"43}
\DeclareMathSymbol{\D}{\mathbin}{AMSb}{"44}
\DeclareMathSymbol{\E}{\mathbin}{AMSb}{"45}
\DeclareMathSymbol{\F}{\mathbin}{AMSb}{"46}
\DeclareMathSymbol{\G}{\mathbin}{AMSb}{"47}
\DeclareMathSymbol{\HH}{\mathbin}{AMSb}{"48}
\DeclareMathSymbol{\I}{\mathbin}{AMSb}{"49}
\DeclareMathSymbol{\N}{\mathbin}{AMSb}{"4E}
\DeclareMathSymbol{\PP}{\mathbin}{AMSb}{"50}
\DeclareMathSymbol{\Q}{\mathbin}{AMSb}{"51}
\DeclareMathSymbol{\R}{\mathbin}{AMSb}{"52}
\DeclareMathSymbol{\SSS}{\mathbin}{AMSb}{"53}
\DeclareMathSymbol{\T}{\mathbin}{AMSb}{"54}
\DeclareMathSymbol{\U}{\mathbin}{AMSb}{"55}
\DeclareMathSymbol{\V}{\mathbin}{AMSb}{"56}
\DeclareMathSymbol{\W}{\mathbin}{AMSb}{"57}
\DeclareMathSymbol{\X}{\mathbin}{AMSb}{"58}
\DeclareMathSymbol{\Y}{\mathbin}{AMSb}{"59}
\DeclareMathSymbol{\Z}{\mathbin}{AMSb}{"5A}

\allowdisplaybreaks[4]

\begin{document}

\title[]
{The positive Schur property on positive projective tensor products and spaces of regular multilinear operators}

\author{Geraldo Botelho}
\address{Faculdade de Matem\'atica, Universidade Federal de Uberl\^andia, Uberl\^andia, MG, Brazil}
\email{botelho@ufu.br}

\author{Qingying Bu}
\address{Department of Mathematics, University of Mississippi, Oxford, MS 38677, USA}
\email{qbu@olemiss.edu}

\author{Donghai Ji}
\address{Department of Mathematics, Harbin University of Science and Technology, Harbin 150080, China}
\email{jidonghai@126.com}

\author{Khazhak Navoyan}
\address{Faculdade de Matem\'atica, Universidade Federal de Uberl\^andia, Uberl\^andia, MG, Brazil}
\email{Khazhakanush@gmail.com}

\date{}

\thanks{The first author is the corresponding author and is supported by CNPq Grant
304262/2018-8 and Fapemig Grant PPM-00450-17.}
\thanks{The fourth author is supported by the Coordena\c c\~ao de Aperfei\c coamento de Pessoal de N\'ivel Superior -- Brasil (CAPES) -- Finance Code 001.}

\maketitle

\begin{abstract} We characterize the positive Schur property in the positive projective tensor products of Banach lattices, we establish the connection with the weak operator topology and we give necessary and sufficient conditions for the space of regular multilinear operators/homogeneous polynomials taking values in a Dedekind complete Banach lattice to have the positive Schur property. 
\end{abstract}

\section{Introduction and background}
A Banach lattice $E$ has the {\it positive Schur property} (PSP in short) if every weakly null sequence formed by positive elements of $E$ is norm null. A lot of research has been done recently on the PSP, see, e.g., \cite{ardakani, baklouti, kaminska, moussa, TRAD, wnuk2009, wnuk2013}. It is worth mentioning that the PSP differs from the usual Schur property for Banach spaces, for instance it is well known that $L_1[0,1]$ has the PSP and fails the Schur property.

Our contribution to the study of the positive Schur property is splitted into three  closely related sections.

In Section \ref{section 2} we caracterize the positive Schur property in the positive projective tensor product $E_1\hat{\otimes}_{|\pi|}\cdots \hat{\otimes}_{|\pi|} E_m$ of the Banach lattices $E_1, \ldots, E_n$ by means of almost Dunford-Pettis operators.

Following the line of a classical result due to Lust (see Theorem \ref{lust}) about the Schur property in the space of weak*-weak continuous operators and recent developments due to Tradacete \cite{TRAD}, in Section \ref{section 3} we establish the connection between the positive Schur property on dual spaces and the weak operator topology. In this section we also show that the space of regular weak*-weak*-continuous linear operators between dual Banach lattices is not a vector lattice with respect to the pointwise ordering.

In Section \ref{section 4} we extend a classical linear result due to Wnuk (see Theorem \ref{wnuk}) to the multilinear and polynomial settings, characterizing the positive Schur property in the spaces of regular multilinear operators and regular homogeneous polynomials taking values in Dedekind complete Banach lattices.




%

For the general theory of Banach lattices we refer to \cite{Mey} 
and for multilinear operators and homogeneous polynomials on Banach spaces we refer to \cite{sean, jorge}. 
Now we recall a few basic notions. For a Banach lattice $E$, $E^+ = \{x \in E : x \geq 0\}$ is the positive cone of $E$ and $E^*$ denotes the topological dual of $E$. Given Banach lattices $E_1, \dots, E_m, E, F$, an $m$-linear operator $T\colon E_1\times \cdots \times E_m \to F$ is said to be:\\
$\bullet$ positive if $T(x_1, \dots, x_m) \in F^+$ whenever $x_1\in E_1^+, \dots, x_m\in E_m^+$;\\
$\bullet$ regular if $T$ is the difference of two positive $m$-linear operators;\\
$\bullet$ an $m$-morphism if $|T(x_1, \dots, x_m)| = T(|x_1|, \dots, |x_m|)$ for all $x_1\in E_1, \dots, x_m\in E_m$.

Let $\mathcal{L}^r(E_1,\dots, E_m; F)$ denote the space of all regular $m$-linear operators from $E_1\times\cdots\times E_m$ to $F$, which becomes a Banach lattice with the regular norm $\|T\|_r = \| \,|T|\, \|$ if $F$ is Dedekind complete (see, e.g., \cite{BuBuskes}). By $E_1\hat{\otimes}_{|\pi|}\cdots \hat{\otimes}_{|\pi|} E_m$ we denote the positive projective tensor product of the Banach lattices $E_1, \ldots, E_m$, that is, the completion of the $m$-fold vector lattice tensor product  $(E_1\Bar{\otimes}\cdots\Bar{\otimes} E_m, \otimes)$ (see \cite{Fr1, fremlin, Schep}) with respect to the positive projective tensor norm defined by
$$
\|u\|_{|\pi|} = \inf\Big\{\sum_{k=1}^n \|x_{1,k}\|\cdots \|x_{m,k}\|: n \in \mathbb{N},\, x_{i,k}\in E_i^+,\, |u| \leq \sum_{k=1}^n x_{1,k}\otimes\cdots\otimes x_{m,k} \Big\},
$$
for every $u \in E_1\Bar{\otimes}\cdots\Bar{\otimes} E_m$. Then $\|\cdot\|_{|\pi|}$ is a lattice norm on $E_1\Bar{\otimes}\cdots\Bar{\otimes} E_m$ (see \cite{BuBuskes, fremlin}), hence $E_1\hat{\otimes}_{|\pi|}\cdots \hat{\otimes}_{|\pi|} E_m$ is a Banach lattice. By $\otimes_m$ we denote the canonical $m$-morphism given by
 $$
\otimes_m \colon E_1\times\cdots\times E_m\longrightarrow E_1\hat{\otimes}_{|\pi|}\cdots \hat{\otimes}_{|\pi|} E_m~,~\otimes_m(x_1,\dots, x_m) = x_1\otimes\cdots\otimes x_m.
$$
 The following result, which follows from \cite[Proposition 3.3]{BuBuskes} combined with \cite[p.\,849 item(c) and (3.2)]{BuBuskes}, is basic for our purposes.

\begin{theorem}\label{basic} Let $E_1, \ldots, E_m$ be Banach lattices. For every  Banach lattice $F$ and any regular $m$-linear operator $T \in {\mathcal L}^r(E_1, \ldots, E_m;F)$ there exists a unique regular linear operator $T^\otimes \in {\mathcal L}^r\left(E_1\hat{\otimes}_{|\pi|}\cdots \hat{\otimes}_{|\pi|} E_m;F \right)$ such that the following diagram is commutative
\begin{displaymath} \hspace{.5 cm}
	\xymatrix{
E_1\times\cdots\times E_m \ar[r]^{\;\;\;\;\;T} \ar[d]_{\otimes_m}\;\;\; & F\\
E_1\hat{\otimes}_{|\pi|}\cdots \hat{\otimes}_{|\pi|} E_m \ar[ur]_{T^\otimes}
}\hspace{.6 cm}
\hspace{.5 cm}
\end{displaymath}
that is,
$$T(x_1, \ldots, x_m) = T^\otimes(x_1 \otimes \cdots \otimes x_m) $$
for all $x_1 \in E_1, \ldots, x_m \in E_m$. Moreover, the correspondence $T \mapsto T^\otimes$ is isometrically isomorphic $\,$ and $\,$ bi-positive $\,$ between $\,$ the $\,$ ordered $\,$ Banach $\,$ spaces $\,$ ${\mathcal L}^r(E_1, \ldots, E_m;F)$ $~$ and $~$ ${\mathcal L}^r\left(E_1\hat{\otimes}_{|\pi|}\cdots \hat{\otimes}_{|\pi|} E_m;F \right)$. If, in addition, $F$ is Dedekind complete, then the correspondence $T \mapsto T^\otimes$ is also lattice homomorphic between the Banach lattices ${\mathcal L}^r(E_1, \ldots, E_m;F)$ and ${\mathcal L}^r\left(E_1\hat{\otimes}_{|\pi|}\cdots \hat{\otimes}_{|\pi|} E_m;F \right)$.
\end{theorem}

 We say that a Banach lattice $E$ is {\it positively isomorphic} to a subspace of the Banach lattice $F$ if there is a positive linear operator $T \colon E \longrightarrow F$ that is a topological isomorphism onto its range. The next straightforward lemma will be helpful several times.

\begin{lemma}\label{nlemma}
Let $E$ and $F$ be Banach lattices such that $E$ is positively isomoprhic to a subspace of $F$. If $F$ has the PSP, then $E$ has the PSP as well.
\end{lemma}


\section{The positive Schur property in the positive projective tensor product}\label{section 2}
It is a long standing problem whether or not the projective tensor product of Banach spaces with the Schur property has the Schur property as well (see, e.g., \cite{pams, gg}). As far as we know, the related problem for Banach lattices has not been investigated yet. Our purpose in this section is to give characterizations of the positive Schur property on the positive projective tensor product of Banach lattices.

To do so, we recall the following class of operators introduced by S\'anchez \cite{sanchez} and developed by several authors, see, e.g. \cite{aqzzouz, retbi, wnuk94}: a linear operator $u \colon E \longrightarrow F$ from the Banach lattice $E$ to the Banach space $F$ is said to be an {\it almost Dunford-Pettis operator} if $(u(x_n))_n$ is norm null in $F$ for every pairwise disjoint weakly null sequence $(x_n)_n$ in $E$.

%

For Banach lattices $E$ and $F$, it is easy to see that the set of almost Dunford-Pettis regular operators from $E$ to $F$  is a closed linear subspace of ${\mathcal L}^r(E;F)$. 

%
%
%
%
%
%

\begin{theorem}
The following are equivalent for the Banach lattices $E_1\dots, E_m$.

\begin{itemize}
    \item [(1)] $E_1\hat{\otimes}_{|\pi|}\cdots \hat{\otimes}_{|\pi|} E_m$ has the PSP.
    \item[(2)] For every Banach lattice $G$ and any $T\in \mathcal{L}^r(E_1,\dots, E_m; G)$, the linearization $T^\otimes \in \mathcal{L}^r(E_1\hat{\otimes}_{|\pi|}\cdots \hat{\otimes}_{|\pi|} E_m; G)$ is an almost Dunford-Pettis operator. 
    \item[(3)] For every Banach lattice $G$ and any $T\in \mathcal{L}^r(E_1,\dots, E_m; G)$, there exists a Banach lattice $H$, an  $m$-linear operator $A\in \mathcal{L}^r(E_1,\dots, E_m; H)$ and an almost Dunford-Pettis operator $u \in \mathcal{L}^r(H; G)$ such that $T = u\circ A$.
\end{itemize}
\end{theorem}

\begin{proof}
$(1)\Rightarrow(2)$ Let $(z_n)_n\subseteq (E_1\hat{\otimes}_{|\pi|}\cdots \hat{\otimes}_{|\pi|} E_m)^+$ be such that $z_n\overset{w}{\longrightarrow} 0$. By the PSP of $E_1\hat{\otimes}_{|\pi|}\cdots \hat{\otimes}_{|\pi|} E_m$ we have $\|z_n\|\longrightarrow 0$, and the continuity of $T^\otimes$ gives $T^\otimes(z_n) \longrightarrow T^\otimes (0) = 0$. By \cite[Theorem 2.2]{aqzzouz} it follows that $T^\otimes$ is an almost Dunford-Pettis operator.

\vspace{.1 in}

$(2)\Rightarrow(3)$ Just take $H = E_1\hat{\otimes}_{|\pi|}\cdots \hat{\otimes}_{|\pi|} E_m$, $A = \otimes_m$ and $u = T^\otimes$.

\vspace{.1 in}

$(3)\Rightarrow(1)$ Take $G = E_1\hat{\otimes}_{|\pi|}\cdots\hat{\otimes}_{|\pi|} E_m$ and consider the canonical $m$-linear regular operator $\otimes_m \in \mathcal{L}^r(E_1,\dots, E_m; E_1\hat{\otimes}_{|\pi|}\cdots\hat{\otimes}_{|\pi|} E_m)$. It is clear that the linearization $(\otimes_m)^\otimes$ of $\otimes_m$ is the identity operator on $E_1\hat{\otimes}_{|\pi|}\cdots\hat{\otimes}_{|\pi|} E_m$ (it follows from the uniqueness in Theorem \ref{basic}).  By $(3)$ there exist a Banach lattice $H$, a regular $m$-linear operator $A\in \mathcal{L}^r(E_1,\dots, E_m; H)$ and an almost Dunford-Pettis regular linear operator $u \in \mathcal{L}^r(H, E_1\hat{\otimes}_{|\pi|}\cdots \hat{\otimes}_{|\pi|} E_m)$ such that $\otimes_m = u\circ A$. For all $x_1\in E_1$,$\ldots$, $x_m\in E_m$,
\begin{align*}(u\circ A)^\otimes (x_1\otimes\cdots\otimes x_m)& = (u\circ A)(x_1,\dots, x_m) = u (A(x_1,\dots, x_m)) \\
&=u(A^\otimes(x_1\otimes\cdots\otimes x_m)) = (u\circ A^\otimes)(x_1\otimes\cdots\otimes x_m).
\end{align*}

$$\begin{tikzcd}
E_1\times\cdots\times E_m \arrow{d}[swap]{\otimes_m} \arrow{r}{A} & H\arrow{r}{u} & G\\
E_1\hat{\otimes}_{|\pi|}\cdots\hat{\otimes}_{|\pi|}E_m \arrow{ur}{A^\otimes} \arrow{urr}[swap]{(u\circ A)^\otimes}
\end{tikzcd}$$

\medskip
Since both $(u\circ A)^\otimes$ and $u\circ A^\otimes$ are regular linear operators, the uniqueness in Theorem \ref{basic} gives
$$
id_{E_1\hat{\otimes}_{|\pi|}\cdots \hat{\otimes}_{|\pi|} E_m} = ({\otimes_m})^\otimes = (u\circ A)^\otimes = u\circ A^\otimes.
$$
Take $(z_n)_n\subseteq E_1\hat{\otimes}_{|\pi|}\cdots\hat{\otimes}_{|\pi|} E_m$ such that $z_n \geq 0$ for every $n\in \N$ and $z_n\overset{w}{\longrightarrow}0$. Continuous linear operators are weak-weak continuous, so
$$(A^\otimes)^+ (z_n)\overset{w}{\longrightarrow} 0{\rm ~and~}(A^\otimes)^- (z_n)\overset{w}{\longrightarrow} 0.$$
Since $(A^\otimes)^+ (z_n)\geq 0$ and $(A^\otimes)^- (z_n)\geq 0$ for every $n$, and $u$ is an almost Dunford-Pettis operator, a second application of \cite[Theorem 2.2]{aqzzouz} gives
$$u((A^\otimes)^+ (z_n))\longrightarrow 0 {\rm ~and~}u((A^\otimes)^- (z_n))\longrightarrow 0, $$
from which it follows that
%
%
%
%
%
%
%
%
%
\begin{align*}
z_n &= id_{E_1\hat{\otimes}_{|\pi|}\cdots\hat{\otimes}_{|\pi|} E_m}(z_n) = u(A^\otimes (z_n)) = u \Big((A^\otimes)^+(z_n) - (A^\otimes)^-(z_n)\Big) \\
&= u\Big ((A^\otimes)^+(z_n)\Big) - u\Big((A^\otimes)^-(z_n)\Big) \longrightarrow 0.
\end{align*}
Hence, $E_1\hat{\otimes}_{|\pi|}\cdots \hat{\otimes}_{|\pi|} E_m$ has the PSP.
\end{proof}

\section{Connections with the weak operator topology}\label{section 3}

 Spaces of weak*-weak linear operators on dual Banach spaces proved to be closely connected to the investigation of the Schur property on Banach and locally convex spaces (see, e.g., \cite{pams, Lust, R3}). The following result due to Lust is illustrative:
\begin{theorem}\cite{Lust}\label{lust} The Banach spaces $X$ and $Y$ have the Schur property if and only if the closed subspace of $\mathcal{L}(X^\ast, Y)$ consisting of all weak*-weak continuous operators (that is, all operators $T\in \mathcal{L}(X^\ast, Y)$ for which $T^\ast(Y^\ast)\subseteq X$) has the Schur property.
\end{theorem}

For Banach lattices, Tradacete \cite{TRAD} introduced the weak operator topology-PSP in ${\mathcal L}^r(E;F)$ and the weak* operator topology-PSP in ${\mathcal L}(E;F^*)$ to study the dual positive Schur property (see \cite[Theorems 2 and 3]{TRAD}).

By $\mathcal{L}^r_{w^\ast}(F^{\ast}, E^\ast)$  we denote the ordered vector subspace of $\mathcal{L}^r(F^{\ast}, E^\ast)$ consisting of all weak*-weak* continuous regular linear operators from the dual $E^*$ of the Banach lattice $E$ to the dual $F^*$ of the Banach lattice $F$. Combining the strategies of Lust and Tradacete, in this section we prove the equivalence between the PSP on $E^*$ and $F^*$ with the weak operator topology-PSP 
on $\mathcal{L}^r_{w^\ast}(F^{\ast\ast};E^\ast)$.


First, let us stress that $\mathcal{L}^r_{w^\ast}(F^{\ast}, E^\ast)$ is not a vector lattice with the pointwise ordering in general.

\begin{example}\rm In order to establish that $\mathcal{L}^r_{w^\ast}(F^{\ast}, E^\ast)$ fails to be a vector lattice with the pointwise ordering in general, we will show that there are Banach lattices $E$ and $F$ and an operator $V \in {\mathcal L}^r(E;F)$ such that $|V^\ast|\notin \mathcal{L}^r_{w^\ast}(F^{\ast}, E^\ast)$. This is enough because $V^\ast \in \mathcal{L}^r_{w^\ast}(F^{\ast}, E^\ast)$.

Our reasoning is based on \cite[Example 1.74]{AB}, where it is claimed that there is an operator $T\colon \ell_1\longrightarrow \ell_\infty$ such that $|T^\ast| < |T|^\ast$. It just so happens that we noticed that there is a gap in this example: in the $13^{th}$ line of the example, we read ``$\langle |T|' \phi, e \rangle = \langle \phi, |T|e\rangle$", while $|T|$ cannot be applied on $e = (1,1,\dots) \notin \ell_1$ in that particular example.


Fortunately, V. G. Troitsky and A. W. Wickstead, to whom we express our thankfulness, constructed correct examples of operators $V\colon c\longrightarrow c$, where $c$ is the closed subspace of $\ell_\infty$ formed by all convergent sequences, such that $|V^\ast| \neq |V|^\ast$ (private communications to the authors). For the benefit of the reader, we describe these two examples.

The Troitsky operator. It is plain that
$$T \colon c\longrightarrow c~,~T(x_1, x_2,\dots) = (x_1 - x_2, x_2 - x_3, x_3 - x_4,\dots).
$$
is a regular operator. Let us prove that $|T|$ exists and
$$
|T|(x_1, x_2,\dots)
= (x_1 + x_2, x_2 + x_3, x_3 + x_4,\dots) {\rm~for~every~} 0 \leq (x_1, x_2, \dots) \in c.$$
Considering the operator
$$S \colon c\longrightarrow c~,~S(x_1, x_2, \dots) = (x_1 + x_2, x_2 + x_3, \dots),$$ we have to prove that $|T|$ exists and $|T| =S$.
Indeed, we have $\pm Tx \leq Sx$ for every $x\in c^+$, so that $\pm T\leq S$. On the
other hand, suppose that $\pm T \leq R$ for some $R \in {\mathcal L}^r(c;c)$. Clearly, $R \geq 0$.
It suffices to show that $R\geq S$, because it would mean that $S = \sup\{T, -T\}$, and the latter is the definition of $|T|$. Note that $\pm Tx\leq Rx$ for every $x\in c^+$, therefore $|Tx| \leq Rx$ for every $x\in c^+$. For a fixed $x = (x_1, x_2, \dots) \in c^+$,
$$
Rx \geq R(x_1, 0, \dots) \geq T(x_1, 0, \dots) = (x_1, 0, \dots).
$$
For every $n > 1$,
\begin{align*}
Rx &\geq R(0,\dots,0,\overset{n}{\overbrace{x_n}}, 0,\dots)\geq |T(0,\dots,0, \overset{n}{\overbrace{x_n}}, 0,\dots)|
= |(0,\dots,0,\overset{n-1}{\overbrace{-x_n}}, \overset{n}{\overbrace{x_n}}, 0,\dots)| \\
 &= (0,\dots,0,\overset{n-1}{\overbrace{x_n}}, \overset{n}{\overbrace{x_n}}, 0,\dots).
\end{align*}
hence,
\begin{align*} Rx &\geq R(0,\dots,0,\overset{n}{\overbrace{x_n}}, \overset{n+1}{\overbrace{x_{n+1}}}, 0,\dots) =R\Big((0,\dots,0,\overset{n}{\overbrace{x_n}}, 0,\dots) + (0,\dots,0, \overset{n+1}{\overbrace{x_{n+1}}}, 0,\dots)\Big)\\
&= R(0,\dots,0,\overset{n}{\overbrace{x_n}}, 0,\dots) + R(0,\dots,0, \overset{n+1}{\overbrace{x_{n+1}}}, 0,\dots)\\
&\geq (0,\dots,0,\overset{n-1}{\overbrace{x_n}}, \overset{n}{\overbrace{x_n}}, 0,\dots) + (0,\dots,0,\overset{n}{\overbrace{x_{n+1}}}, \overset{n+1}{\overbrace{x_{n+1}}}, 0,\dots)\\
&= (0,\dots,0,\overset{n-1}{\overbrace{x_n}}, \overset{n}{\overbrace{x_n + x_{n+1}}}, \overset{n+1}{\overbrace{x_{n+1}}}, 0,\dots)\geq (0,\dots,0, \overset{n}{\overbrace{x_n + x_{n+1}}}, 0,\dots) =: y^{(n)}.
\end{align*}
It is easy to see that the sequence $(y^{(n)})_n$ has supremum in $c$, which
is exactly $Sx$. It follows that $Rx \geq Sx$ for every $ x\in c^+$, therefore $R \geq S$. This proves the claim that $|T| = S$.

Let $\varepsilon_\infty \in (c^{\ast})^+$ be defined by
\begin{equation} \label{linfunct}
\varepsilon_\infty(x_1, x_2, \dots) = \lim_{n \to \infty} x_n {\rm ~~for~every~} x = (x_1, x_2,\dots) \in c.
\end{equation}
Fix $e: = (1, 1, \ldots) \in c$ and note that
$$
\langle |T|^\ast \varepsilon_\infty, e\rangle = \langle \varepsilon_\infty, |T|e\rangle = \varepsilon_\infty(2,2,\dots) = 2.
$$
Since $c^\ast$ is a dual space, it is Dedekind complete and hence the Riesz-Kantorovich formula yields that
$$
|T^\ast|\varepsilon_\infty = \sup\{|T^\ast \psi|: |\psi|\leq \varepsilon_\infty\}.
$$
Let $\psi \in c^\ast$ satisfy $|\psi|\leq \varepsilon_\infty$ and let $x = (x_1, x_2,\dots) \in c^+$. Then,
\begin{align*}|\langle T^\ast \psi,x\rangle| &= |\langle \psi, T(x) \rangle| \leq \langle|\psi|, |Tx| \rangle\leq \langle \varepsilon_\infty, |Tx|  \rangle
= \langle\varepsilon_\infty, (|x_1 - x_2|, |x_2 - x_3|,\dots) \rangle \\
&= \lim_{n\to\infty}|x_n - x_{n+1}| = 0.
\end{align*}
Therefore, $T^\ast \psi = 0$, and hence $|T^\ast| \varepsilon_\infty = 0$.
In particular, $\langle |T^\ast| \varepsilon_\infty, e \rangle = 0 \neq 2 = \langle|T|^\ast \varepsilon_\infty, e\rangle$, proving that $|T^\ast| < |T|^\ast$.

The Wickstead operator. Again, it is plain that
$$W\colon  c\longrightarrow c~,~Wx = W(x_1, x_2, \dots) = x - e \cdot\lim_{k\to\infty} x_k = (x_k - \lim_{j\to\infty} x_j)_{k=1}^\infty,$$
is a regular operator. For any $x = (x_k)_{k=1}^\infty = (x_1, x_2, \dots) \in c^+$,
\begin{align}x = (x_1, x_2, \dots) & = \sup\left\{y = (y_1, y_2, \dots)\in c: 0\leq y\leq x {\rm ~ and~} \lim_{k\to \infty} y_k = 0\right\} \label{eqqe} \\
&= \sup\{y = (y_1, y_2, \dots)\in c_0: 0\leq y\leq x\}.\nonumber
\end{align}
Denoting by $I$ the identity operator on $c$, $I \geq W$ and $I \geq 0$. On the other hand, if $U\geq W, U \geq 0$, that is, if $U$ is an arbitrary upper bound for $\{W, 0\}$,  and $x = (x_1, x_2,\dots) \geq 0$, then
$$
 Uy \geq Wy {\rm ~and~} Ux \geq Uy {\rm ~for~every~} y \in c {\rm~with~}  0 \leq y \leq x.
$$
Therefore,
$
Ux \geq Wy$ for every $y \in c$ with $0 \leq y \leq x$, implying that $Ux$ is an upper bound for $\{Wy: y\in c,\; 0\leq y\leq x\}$.
So, $Ux$ is also an upper bound for
$$\left\{Wy: y = (y_k)\in c,\; 0\leq y\leq x {\rm ~and~} \lim_{k\to\infty} y_k = 0\right\} = \left\{y: y \in c,\; 0\leq y\leq x {\rm ~ and~} \lim_{k\to \infty} y_k = 0\right\},$$
where the last equality holds because $Wy = y - \lim\limits_{k\to \infty} y_k \cdot e = y$ whenever $\lim\limits_{k\to \infty} y_k = 0$.
From (\ref{eqqe}) we get that $Ux \geq x = Ix$ any $x\in c^+$, proving that $U\geq I$.
It follows that $W^+$ exists and is equal to $I$, because above it was observed that $I$ is an upper bound for $\{W, 0\}$, and the arbitrary upper bound $U$ of $\{W, 0\}$ is greater than $I$. Thus, $I = \sup\;\{W, 0\} = W^+$.

Considering again the linear functional $\varepsilon_\infty \in (c^*)^+$ (see (\ref{linfunct})),
 we have
$$W^\ast (\varepsilon_{\infty}) (x) = \varepsilon_{\infty} (Wx) = \lim_{k\to\infty} \left\{\left(x_k-\lim_{j\to\infty} x_j\right)_{k=1}^\infty\right\}= 0$$ for all $x = (x_1, x_2, \dots)\in c$, so $W^\ast (\varepsilon_\infty) =0$.
Since $c^\ast$ is Dedekind complete and $\varepsilon_\infty$ is an atom,
$$(W^\ast)^+ (\varepsilon_\infty) = \sup\{W^\ast \mu : 0 \leq \mu \leq \varepsilon_\infty\} = 0.$$
As $W^+ = I$, in particular we have $(W^+)^\ast (\varepsilon_\infty) = I^\ast (\varepsilon_\infty) = \varepsilon_\infty \neq 0$, hence $(W^\ast)^+ \neq (W^+)^\ast$ and it follows that $|W^\ast| \neq |W|^\ast$.


Let $V$ be either the Troitsky operator $T$ or the Wickstead operator $W$. Suppose that $|V^*| \in \mathcal{L}^r_{w^\ast}(F^{\ast}, E^\ast)$. In this case there exists an operator $S \colon c\longrightarrow c$ such that $|V^\ast| = S^\ast$. 
Then, $\pm V^{\ast\ast} \leq S^{\ast\ast}$, so restricting these second adjoints to $E$ we get that $\pm V \leq S$, hence $|V| \leq S$. Consider the canonical unit vectors $(e_i)_{i}$ and note that, for all $ x\in E^+$ and $i \in \mathbb{N}$,
$$
(Sx)_i = \langle e_i^\ast, Sx\rangle = \langle S^\ast e_i^\ast, x \rangle = \langle |V^\ast| e_i^\ast, x \rangle
\leq \langle |V|^\ast e_i^\ast, x \rangle = \langle e_i^\ast, |V|x \rangle = (|V|x)_i,
$$
from which it follows that $S(x)\leq |V|(x)$ for every $ x\in E^+$, that is, $S\leq |V|$.
From $S = |V|$ we conclude that $ |V|^\ast = S^\ast = |V^\ast|$, a contradiction that shows that $|V^*| \notin \mathcal{L}^r_{w^\ast}(F^{\ast}, E^\ast)$. Since $V^* \in \mathcal{L}^r_{w^\ast}(F^{\ast}, E^\ast)$, it follows that $\mathcal{L}^r_{w^\ast}(F^{\ast}, E^\ast)$ fails to be a vector lattice.
\end{example}

For Banach spaces $E$ and $F$, we recall that a sequence of operators $(T_n)\subseteq \mathcal{L}(E;F)$ converges to zero in the {\it weak operator topology} (wot) if $\langle y^\ast, T_n x \rangle \longrightarrow 0$ for all $x\in E$ and $y^\ast \in F^\ast$.

The connection of the PSP with the space of weak*-weak* continuous operators is made by the following concept introduced by Tradacete:


\begin{definition}\cite{TRAD} \rm Let $E$ and $F$ be Banach lattices.
A subspace $X$ of $\mathcal{L}^r(E; F)$ has the {\it wot-PSP} if for every sequence of positive operators $(T_n)\subseteq X^+$ with $T_n\longrightarrow 0$ in the weak operator topology, it follows that $\|T_n\|\longrightarrow 0$.
\end{definition}

\begin{theorem}\label{wot}
The following are equivalent for the Banach lattices $E$ and $F$:

\begin{enumerate}
\item $E^\ast$ and $F^\ast$ have the PSP.
\item $\mathcal{L}^r_{w^\ast}(F^{\ast\ast};E^\ast)$ has the wot-PSP.
\item $(E\hat{\otimes}_{|\pi|}F)^\ast$ has the PSP.
\end{enumerate}
\end{theorem}

\begin{proof} $ (1)\Rightarrow (2)$ Let $(S_n)_n\subseteq$ $\mathcal{L}^r_{w^\ast}(F^{\ast\ast};E^\ast)$ be a sequence of positive weak$^\ast$-to-weak$^\ast$ continuous operators such that $S_n \longrightarrow 0$ in the weak operator topology. Then, for every $n \in \mathbb{N}$ there exists $T_n \in \mathcal{L}^r(E;F^\ast)$ such that $S_n=T_n^\ast$, and $\langle x^{\ast\ast}, S_n( y^{\ast\ast}) \rangle \longrightarrow 0$ for all $x^{\ast\ast} \in E^{\ast\ast}$, $ y^{\ast\ast} \in F^{\ast\ast}$. Note that $\lVert S_n\rVert = \lVert T_n^\ast\rVert=\lVert T_n \rVert$ for every $n$.

Suppose that $(S_n)_n$ is not norm null. In this case there exists $\alpha > 0$  such that (up to a subsequence if necessary) $\lVert S_n\rVert=\lVert T_n \rVert> \alpha$ for every $n$. Then, for  every $n \in \mathbb{N}$ there exists   $x_n \in B_{E^+}$ such that $\lVert T_n( x_n)\rVert \geq \alpha$. For all $x^{\ast\ast} \in E^{\ast\ast}$ and $y^{\ast\ast} \in F^{\ast\ast}$,
$$\langle x^{\ast\ast}, S_n (y^{\ast\ast})\rangle=\langle x^{\ast\ast}, T_n^\ast (y^{\ast\ast})\rangle \longrightarrow 0,$$
which implies that $T_n^\ast (y^{\ast\ast})\overset{w}{\longrightarrow} 0$ in $E^\ast$. By the PSP of $E^\ast$ we have
$$\lVert T_n^\ast (y^{\ast\ast})\rVert = \lVert T_n^\ast ((y^{\ast\ast})^+ - (y^{\ast\ast})^-) \rVert \leq \lVert T_n^\ast y^{\ast\ast+}\rVert+\lVert T_n^\ast y^{\ast\ast-}\rVert\longrightarrow 0.$$
Thus, for every $y^{\ast\ast} \in F^{\ast\ast}$,
$$
|\langle y^{\ast\ast}, T_n (x_n)\rangle| = |\langle T_n^\ast( y^{\ast\ast}), x_n\rangle|\leq
\lVert T_n^\ast( y^{\ast\ast})\rVert \cdot \lVert x_n\rVert\longrightarrow 0.
$$
It follows that $T_n x_n\overset{w}{\longrightarrow} 0$ in $F^\ast$, from which the PSP of $F^\ast$ yields that $\lVert T_n x_n\rVert\longrightarrow 0$. This is a contradiction because $\lVert T_n x_n\rVert\geq \alpha$ for every $n\in \N$. Therefore, $\|S_n\| \longrightarrow 0$, proving that $\mathcal{L}^r_{w^\ast}(F^{\ast\ast};E^\ast)$ has the wot-PSP.

\vspace{.1 in}

$(2)\Rightarrow (3)$ The assumption is that $\mathcal{L}^r_{w^\ast}(F^{\ast\ast};E^\ast)$ has the wot-PSP. Take $(\phi_n)_n\subseteq (E\hat{\otimes}_{|\pi|}F)^*$ a sequence of positive functionals such that $(\phi_n)$  is weakly convergent to zero. We have already used that $(E\hat{\otimes}_{|\pi|}F)^\ast$, ${\mathcal L}^r(E,F, \mathbb{R})$ and $\mathcal{L}^r(E;F^\ast)$ are isometrically isomorphic and lattice homomorphic.   Let $(T_n)_n$ be the sequence of the corresponding operators for $(\phi_n)$, that is, each $T_n \in \mathcal{L}^r(E;F^\ast)$ and
$$\langle T_n( x), y\rangle = \phi_n(x\otimes y) {\rm ~for~all~} x \in E {\rm ~and~} y \in F.$$
Given $x^{\ast\ast} \in E^{\ast\ast}$ and  $y^{\ast\ast} \in F^{\ast\ast}$, define $$\xi_{x^{\ast\ast}, y^{\ast\ast}} \colon (E\hat{\otimes}_{|\pi|} F)^{\ast}\longrightarrow \mathbb{R}~,~
\xi_{x^{\ast\ast}, y^{\ast\ast}}(\phi) = \langle T_\phi^{\ast\ast}(x^{\ast\ast}), y^{\ast\ast}\rangle.$$
It is easy to check that $\xi_{x^{\ast\ast}, y^{\ast\ast}} \in (E\hat{\otimes}_{|\pi|} F)^{\ast\ast}$.
As $0\leq\phi_n\overset{w}{\longrightarrow} 0$ in $(E\hat{\otimes}_{|\pi|}F)^\ast$, we have
$$
\xi_{x^{\ast\ast}, y^{\ast\ast}}(\phi_n) = \langle T_{\phi_n}^{\ast\ast} (x^{\ast\ast}), y^{\ast\ast} \rangle = \langle x^{\ast\ast}, T_{\phi_n}^\ast (y^{\ast\ast} )\rangle\longrightarrow 0,
$$
for all $x^{\ast\ast} \in E^{\ast\ast}$ and $y^{\ast\ast} \in F^{\ast\ast}$.
This means that $\left(T_{\phi_n}^\ast\right)_n$ is convergent to zero in weak operator topology in $\mathcal{L}^r_{w^\ast}(F^{\ast\ast}; E^\ast)$. The wot-PSP of the space $\mathcal{L}^r_{w^\ast}(F^{\ast\ast}; E^\ast)$ gives $\|T_{\phi_n}^\ast\|\longrightarrow 0$. Since
$$\|T_{\phi_n}^\ast\| = \|T_{\phi_n}\| = \|\;|T_{\phi_n}|\;\| = \|T_{\phi_n}\|_r = \|\phi_n\|$$
for every $n$, we conclude that $\|\phi_n\|\longrightarrow 0$, which proves that $(E\hat{\otimes}_{|\pi|} F)^\ast$ has the PSP.

\medskip

$ (3)\Rightarrow (1)$ Fixed $0 \neq y^* \in (F^*)^+$, it is clear that $E^*$ is positively isometric to a subspace of ${\mathcal L}^r(E,F;\mathbb{R})$ via the embedding
$$x^* \in E^* \mapsto x^* \otimes y^* \in  {\mathcal L}^r(E,F;\mathbb{R})~,~(x^* \otimes y^*)(x,y) = x^*(x)y^*(y).$$
The same holds for $F^*$. Since ${\mathcal L}^r(E,F;\mathbb{R})$ is isometrically isomorphic and lattice homomorphic to $(E\hat{\otimes}_{|\pi|}F)^\ast$ (Theorem \ref{basic}), the implication follows Lemma \ref{nlemma}.
%
\end{proof}


\section{Spaces of regular multilinear operators and homogeneous polynomials}\label{section 4}

In the Banach space setting, Ryan proved that the space of bounded linear operators from the Banach space $E$ to the Banach space $F$ has the Schur property if and only if the dual $E^*$ of $E$ and $F$ have the Schur property \cite[Theorem 3.3]{R3}. The extension of this result to the multilinear and polynomial cases was done in \cite{pams}.

The lattice counterpart of Ryan's linear result was settled by Wnuk:
\begin{theorem} \cite[Theorem 3]{W1}\label{wnuk}
Let $E$ and $F$ be Banach lattices with $F$ Dedekind complete. The space $\mathcal{L}^r(E;F)$ of regular operators from $E$ to $F$ has the positive Schur property if and only if $E^\ast$ and $F$ have the positive Schur property.
\end{theorem}

In this section we extend Wnuk's result to the multilinear and polynomial cases, that is, we characterize the positive Schur property on the Banach lattices ${\mathcal L}^r(E_1, \ldots, E_m;F)$ of regular $m$-linear operators from $E_1 \times \cdots \times E_m$ to $F$ and ${\mathcal P}^r(^mE;F)$ of regular $m$-homogeneous polynomials from $E$ to $F$, where $F$ is Dedekind complete. We end up with characterizations similar to \cite[Proposition 4.3]{pams}, for example, in Theorem \ref{main} we prove that ${\mathcal L}^r(E_1, \ldots, E_m;F)$ has the PSP if and only if $E_1^*, \ldots, E_m^*$ and $F$ have the PSP.

We start by characterizing the positive Schur property in spaces of regular bilinear operators.

\begin{theorem}\label{bilinear}
Let $E_1, E_2, F$ be Banach lattices with $F$ Dedekind complete.
Then, the Banach lattice of regular bilinear operators $\mathcal{L}^r(E_1, E_2;F)$ has the PSP if and only if $E_1^\ast, E_2^\ast, F$ have the PSP.
\end{theorem}

\begin{proof}
Suppose $E_1^\ast, E_2^\ast, F$ have the PSP. Theorem \ref{wnuk} gives that $\mathcal{L}^r(E_1,E_2^\ast)$ has the PSP. Since $(E_1\hat{\otimes}_{|\pi|} E_2)^{\ast} = \mathcal{L}^r(E_1,E_2^\ast)$ lattice isometrically, 
it follows that $(E_1\hat{\otimes}_{|\pi|} E_2)^{\ast}$ has the PSP.
Now that we know that $(E_1\hat{\otimes}_{|\pi|} E_2)^{\ast}$ and $F$ have the PSP and $F$ is Dedekind complete, Theorem \ref{wnuk} gives that  $\mathcal{L}^r(E_1\hat{\otimes}_{|\pi|}E_2;F)$ has the PSP, hence $\mathcal{L}^r(E_1,E_2;F)$ has the PSP by Theorem \ref{basic}.

The converse may be proved applying Theorem \ref{wnuk} twice, but we'd rather give an elementary reasoning. Suppose that $\mathcal{L}^r(E_1, E_2;F)$ has the PSP. Fix $0 \neq \varphi_2 \in (E_2^*)^+$ and $0 \neq y \in F^+$. It is easy to see that $E_1^*$ is positively isomorphic to a subspace of $\mathcal{L}^r(E_1, E_2;F)$ via the operator
$$T \colon E_1^* \longrightarrow \mathcal{L}^r(E_1, E_2;F)~,~T(\varphi_1)(x_1, x_2) = \varphi_1(x_1)\varphi_2(x_2)y. $$
By Lemma \ref{nlemma}, $E_1^*$ has the PSP and, analogously, $E_2^*$ has the PSP as well. Fix $0 \neq \varphi_j \in (E_j^*)^+, j = 1,2,$ and note that $F$ is  positively isomorphic to a subspace of $\mathcal{L}^r(E_1, E_2;F)$ via the operator
$$T \colon F \longrightarrow \mathcal{L}^r(E_1, E_2;F)~,~T(y)(x_1, x_2) = \varphi_1(x_1)\varphi_2(x_2)y. $$
The PSP of $F$ follows from Lemma \ref{nlemma}.
\end{proof}

Now we turn to the multilinear case.

\begin{proposition}\label{tthh}
For Banach lattices $E_1, \ldots, E_m$,  $(E_1\hat{\otimes}_{|\pi|}\cdots \hat{\otimes}_{|\pi|} E_m)^\ast$ has the PSP if and only if $E_1^\ast,  \ldots, E_m^\ast$ have the PSP.
\end{proposition}

\begin{proof}
First we assume that $E_1^\ast, E_2^\ast, \ldots, E_m^\ast$ have the PSP. By Theorem \ref{wot}, $(E_1\hat{\otimes}_{|\pi|} E_2)^\ast$ has the PSP. Now suppose that  $(E_1\hat{\otimes}_{|\pi|}\cdots\hat{\otimes}_{|\pi|} E_n)^\ast$ has the PSP for some $n < m$. Since $(E_1\hat{\otimes}_{|\pi|}\cdots\hat{\otimes}_{|\pi|} E_n)^\ast$ and $E_{n+1}^\ast$ have the PSP, calling on Theorem \ref{wot} once again and using the associativity of the positive project tensor product \cite[Corollary 1G]{fremlin}, we conclude that $$(E_1\hat{\otimes}_{|\pi|}\cdots\hat{\otimes}_{|\pi|} E_n \hat{\otimes}_{|\pi|} E_{n+1})^\ast = \left((E_1\hat{\otimes}_{|\pi|}\cdots\hat{\otimes}_{|\pi|} E_n) \hat{\otimes}_{|\pi|} E_{n+1}\right)^\ast $$ has the PSP.

 Conversely, assume that $(E_1\hat{\otimes}_{|\pi|}\cdots \hat{\otimes}_{|\pi|} E_m)^\ast$ has the PSP. Up to isometric isomorphisms and lattice homomorphisms,
\begin{align*} {\mathcal L}^r(E_1; (E_2\hat{\otimes}_{|\pi|}\cdots \hat{\otimes}_{|\pi|} E_m)^\ast)& = {\mathcal L}^r (E_1; {\mathcal L}^r(E_2, \ldots, E_m;\mathbb{R})) = {\mathcal L}^r (E_1,E_2, \ldots, E_m;\mathbb{R})\\
& = (E_1\hat{\otimes}_{|\pi|}\cdots \hat{\otimes}_{|\pi|} E_m)^\ast,
\end{align*}
so, ${\mathcal L}^r(E_1; (E_2\hat{\otimes}_{|\pi|}\cdots \hat{\otimes}_{|\pi|} E_m)^\ast)$ has the PSP. By Theorem \ref{wnuk}, $E_1^*$ and $(E_2\hat{\otimes}_{|\pi|}\cdots \hat{\otimes}_{|\pi|} E_m)^\ast$ have the PSP. Starting now with $(E_2\hat{\otimes}_{|\pi|}\cdots \hat{\otimes}_{|\pi|} E_m)^\ast$, a repetition of this reasoning gives that $E_2^*$ has the PSP, and so on. 
\end{proof}

Next we have the multilinear version of Theorem \ref{wnuk}.


\begin{theorem}\label{main}
Let $E_1, \dots, E_m, F$ be Banach lattices with F Dedekind complete. Then $\mathcal{L}^r(E_1,\dots,E_m;F)$ has the PSP if and only if $E_1^\ast, \dots, E_m^\ast$ and $F$ have
the PSP.
\end{theorem}

\begin{proof} We just have to combine some of the previous results:
\begin{align*} E_1^\ast, \dots, E_m^\ast {\rm ~and~}F {\rm ~have~the~PSP} & \Longleftrightarrow (E_1\hat{\otimes}_{|\pi|}\cdots \hat{\otimes}_{|\pi|} E_m)^\ast {\rm ~and~}F {\rm ~have~the~PSP}\\
& \Longleftrightarrow \mathcal{L}^r(E_1\hat{\otimes}_{|\pi|}\cdots \hat{\otimes}_{|\pi|} E_m;F) {\rm ~has~the~PSP}\\
& \Longleftrightarrow \mathcal{L}^r(E_1,\ldots, E_m; F) {\rm ~has~the~PSP},
\end{align*}
where the first equivalence follows from Proposition \ref{tthh}, the second from Theorem \ref{wnuk} and the third from Theorem \ref{basic}.
\end{proof}

By $\mathcal{P}^r(^mE; F)$ we denote the Banach lattice of regular continuous $m$-homogeneous polynomials from $E$ to $F$ (see \cite{BuBuskes}). We write $\mathcal{L}^r(^mE;F)$, $\mathcal{L}^r(^mE)$  and $\mathcal{P}^r(^mE)$ instead of $\mathcal{L}^r(E, \stackrel{(m)}{\ldots},E;F)$, $\mathcal{L}^r(^mE;\mathbb{R})$ and $\mathcal{P}^r(^mE;\mathbb{R})$, respectively.

\begin{theorem}
Let $E$ and $F$ be Banach lattices with $F$ Dedekind complete. Then the following are equivalent.
\begin{itemize}
    \item [(1)] $E^\ast$ and $F$ have the PSP.
    \item[(2)] $\mathcal{L}^r(E;F)$ has the PSP.
    \item[(3)] $\mathcal{L}^r(^mE;F)$ has the PSP for every $ m \in \mathbb{N}$.
    \item[(4)] $\mathcal{L}^r(^mE;F)$ has the PSP for some $ m \in \mathbb{N}$.
    \item[(5)] $\mathcal{P}^r(^mE;F)$ has the PSP for every $ m \in \mathbb{N}$.
    \item[(6)] $\mathcal{P}^r(^mE;F)$ has the PSP for some $ m \in \mathbb{N}$.
    \item[(7)] $\mathcal{L}^r(^mE)$ has the PSP for every $ m \in \mathbb{N}$ and $F$ has the PSP.
    \item[(8)] $\mathcal{L}^r(^mE)$ has the PSP for some $ m \in \mathbb{N}$ and $F$ has the PSP.
    \item[(9)] $\mathcal{P}^r(^mE)$ has the PSP for every $ m \in \mathbb{N}$ and $F$ has the PSP.
    \item[(10)] $\mathcal{P}^r(^mE)$ has the PSP for some $ m \in \mathbb{N}$ and $F$ has the PSP.
\end{itemize}
\end{theorem}
\begin{proof}
$(1)\Leftrightarrow(2)$ follows from Theorem \ref{wnuk}. From \cite[Theorem 3.3]{BuBuskes} we have that $\mathcal{L}^r(^mE; F)$ is isometrically isomorphic and lattice homomorphic to $\mathcal{L}^r (E\hat{\otimes}_{|\pi|}\stackrel{(m)}{\cdots} \hat{\otimes}_{|\pi|}E; F) = \mathcal{L}^r (\hat{\otimes}_{m, |\pi|}E; F)$,
so $(1)\Leftrightarrow(3)$ and $(1)\Leftrightarrow(4)$ follow from Theorem \ref{main}. 
Thus, $(1) - (4)$ are equivalent.

From inequalities  (2.13) in \cite{BuBuskes} we know that the correspondence that associates a polynomial $P \in \mathcal{P}^r(^mE;F)$ to its (unique) associated symmetric $m$-linear operator ${\check P} \in \mathcal{L}^r(^mE;F)$ is a positive isomorphism into. So, $\mathcal{P}^r(^mE;F)$ is positively isomorphic to a subspace of $\mathcal{L}^r(^mE;F)$.  Lemma \ref{nlemma} gives $(3)\Rightarrow(5)$ and $(4)\Rightarrow(6)$.

Using again that $\mathcal{L}^r(^mE; F)$ is isometrically isomorphic and lattice homomorphic to $\mathcal{L}^r (\hat{\otimes}_{m, |\pi|}E; F)$ and calling on \cite[Theorem 3.3]{BuBuskes} to get that $\left(\hat{\otimes}_{m, |\pi|}E \right)^*$ is isometrically isomorphic and lattice homomorphic to ${\mathcal L}^r(^mE)$, from Theorem \ref{wnuk} we get $(3)\Leftrightarrow(7)$ and $(4)\Leftrightarrow(8)$.

Let $\hat{\otimes}_{m,s, |\pi|}E $ denote the positive $m$-fold symmetric projective tensor product of $E$ (see \cite{BuBuskes}). By \cite[Proposition 3.4]{BuBuskes} we know that $\mathcal{P}^r(^mE; F)$ is isometrically isomorphic and lattice homomorphic to  $\mathcal{L}^r (\hat{\otimes}_{m,s, |\pi|}E; F)$ and, in particular, $\left(\hat{\otimes}_{m,s, |\pi|}E \right)^*$ is isometrically isomorphic and lattice homomorphic to $ {\mathcal P}^r(^mE)$. Theorem \ref{wnuk} gives $(5)\Leftrightarrow(9)$ and $(6)\Leftrightarrow(10)$.\\
\indent Since $(9)\Rightarrow(10)$ is obvious, all that is left to prove is $(10)\Rightarrow(1)$. To prove this implication, fix $\gamma \in (E^*)^+$, $\|\gamma\| = 1$, and consider the operator
$$T \colon E^* \longrightarrow {\mathcal P}^r(^mE)~,~T(\varphi)(x) = \varphi(x)\gamma(x)^{m-1}.  $$
In \cite[Corollary 4.2]{preprint} it is proved that $T$ is a positive isomorphism onto its range, that is, $E^*$ is positively isomorphic to a subspace of ${\mathcal P}^r(^mE)$. Lemma \ref{nlemma} completes the proof.
\end{proof}

\begin{corollary} The following are equivalent for a Banach lattice $E$:
\begin{itemize}
    \item [(1)] $E^\ast$ has the PSP.

    \item[(2)] $\mathcal{L}^r(^mE)$ has the PSP for every $ m \in \mathbb{N}$.
    \item[(3)] $\mathcal{L}^r(^mE)$ has the PSP for some $ m \in \mathbb{N}$.
    \item[(4)] $\mathcal{P}^r(^mE)$ has the PSP for every $ m \in \mathbb{N}$.
    \item[(5)] $\mathcal{P}^r(^mE)$ has the PSP for some $ m \in \mathbb{N}$.
\end{itemize}
\end{corollary}

\medskip

\noindent{\bf Acknowledgements.} The authors thank Vladimir G. Troitsky, Anthony W. Wickstead and Jos\'e Lucas P. Luiz for their very important suggestions.
\bibliographystyle{amsplain}

\end{document}